\documentclass[11pt]{amsart}
\usepackage{amssymb, amsmath, amscd, eucal, rotating}
\usepackage{pstricks,pstricks-add,pstcol}
\usepackage{color}

\definecolor{darkgreen}{rgb}{0, 0.45, 0}
\definecolor{lightgreen}{rgb}{0.5, 1, 0.5}

\input xy
 \xyoption{all}

\numberwithin{equation}{section}

\theoremstyle{plain}

\newtheorem{proposition}{Proposition}[section]
\newtheorem{theorem}[proposition]{Theorem}		
		
\newtheorem{corollary}[proposition]{Corollary}
\newtheorem{lemma}[proposition]{Lemma}

\theoremstyle{definition}

\newtheorem{definition}[proposition]{Definition}
\newtheorem{remark}[proposition]{Remark}
\newtheorem{example}[proposition]{Example}

\newcommand{\C}{\mathbb C}
\newcommand{\R}{\mathbb R}

\newcommand{\Gr}{\mathop{\rm Gr}\nolimits}

\newcommand{\End}{\mathop{\rm End}\nolimits}

\newcommand{\dist}{\mathop{\rm dist}\nolimits}

\newcommand{\Ad}{\mathop{\rm Ad}\nolimits}

\newcommand{\tr}{\mathop{\rm Tr}\nolimits}

\newcommand{\SL}{\mathsf{SL}}

\newcommand{\GL}{\mathsf{GL}}

\newcommand{\U}{\mathsf{U}}

\DeclareMathOperator{\id}{id}

\DeclareMathOperator{\rank}{rank}
\DeclareMathOperator{\slope}{slope}

\DeclareMathOperator{\Res}{Res}
\DeclareMathOperator{\cotanh}{cotanh}

\begin{document}


\title{Analytic convergence of harmonic metrics for parabolic Higgs bundles}

\author{Semin Kim}
\address{Department of Mathematics, Brown University, Providence, RI 02912}
\email{smkim@math.brown.edu}

\author{Graeme Wilkin}
\address{Department of Mathematics, National University of Singapore, Singapore 119076}
\email{graeme@nus.edu.sg}

\thanks{This research was partially supported by grant number R-146-000-200-112 from the National University of Singapore. The authors also acknowledge support from NSF grants DMS 1107452, 1107263, 1107367 ``RNMS GEometric structures And Representation varieties'' (the GEAR Network). The first author was also supported in part by the Samsung Scholarship.}

\begin{abstract}
In this paper we investigate the moduli space of parabolic Higgs bundles over a punctured Riemann surface with varying weights at the punctures. We show that the harmonic metric depends analytically on the weights and the stable Higgs bundle. This gives a Higgs bundle generalisation of a theorem of McOwen on the existence of hyperbolic cone metrics on a punctured surface within a given conformal class, and a generalisation of a theorem of Judge on the analytic parametrisation of these metrics.
\end{abstract}

\date{\today}





\maketitle


\thispagestyle{empty}

\baselineskip=16pt



\section{Introduction}

The uniformisation theorem shows that any compact Riemann surface admits a metric of constant scalar curvature within each conformal class. One way of proving this is to solve the resulting partial differential equation for the conformal factor, which was carried out by Berger when the Euler characteristic is nonpositive \cite{Berger69}. More generally, by solving this PDE with a different curvature function, Kazdan and Warner \cite{KazdanWarner74-compact} gave necessary and sufficient conditions for a given function to be equal to the scalar curvature of some metric within a given conformal class, generalising the sufficient conditions given by Berger \cite{Berger71}.

It was originally observed by Hitchin \cite{Hitchin87} that this PDE can be solved in the general framework of the Hitchin-Kobayashi correspondence for Higgs bundles. The theory of Hitchin \cite{Hitchin87} and Simpson \cite{Simpson88} shows that each stable Higgs bundle admits a unique metric solving the self-duality equations. Hitchin observed in \cite[Sec. 11]{Hitchin87} that, for a particular example of stable Higgs bundle (a \emph{Fuchsian point} in the moduli space), the metric solving the self-duality equations is related to a metric solving the constant scalar curvature equations on the underlying compact Riemann surface. Moreover, by deforming the Higgs field one can obtain all of the constant curvature metrics on the underlying smooth surface, which leads to Hitchin's construction of Teichm\"uller space for genus $g \geq 2$. The key idea is to find a stable Higgs bundle for which the Hitchin-Kobayashi correspondence produces the required metric, instead of solving the PDE for the metric directly.

Subsequently Simpson \cite{Simpson90} showed that a stable parabolic Higgs bundle with regular singularities admits a unique metric solving the self-duality equations on a punctured surface. An important aspect of the theory for noncompact surfaces is the need to control the growth of the metric near the punctures. These growth conditions are determined by the stability condition in the form of weights (introduced by Mehta and Seshadri \cite{MehtaSeshadri80} for parabolic bundles without a Higgs field) and by the weight filtration, which depends on the residue of the Higgs field at the punctures. Biswas, Ar{\'e}s-Gastesi and Govindarajan \cite{Biswas97-1} showed that Simpson's theory for parabolic Higgs bundles can be used to produce constant curvature cusp metrics on punctured surfaces, and they extended this construction to the higher Teichm\"uller theory introduced by Hitchin \cite{Hitchin92}.

With respect to metrics with a conical singularity at the punctures, results of McOwen \cite{McOwen88} prove the existence of hyperbolic cone metrics of any cone angle within a given conformal class, and subsequent work of Judge \cite{Judge98} shows that these metrics depend analytically on the cone angles. Both of these proofs are in the spirit of the approach of Berger \cite{Berger69}, \cite{Berger71} and Kazdan-Warner \cite{KazdanWarner74-compact} which involves solving a PDE for the conformal factor. Troyanov \cite{Troyanov91} also used this approach to give necessary and sufficient conditions for a function to be the curvature of a metric with conical singularities within a given conformal class. A Higgs bundle approach to constructing cone metrics for the case where the cone angles of the form $\frac{2\pi}{m}$ for $m \in \mathbb{Z}$ was developed by Nasatyr and Steer \cite{NasatyrSteer95}, who studied orbifold Higgs bundles on a finite ramified cover of the underlying compact Riemann surface. Nasatyr and Steer also proved a connection between orbifold Higgs bundles and parabolic Higgs bundles with rational weights and trivial weight filtration.

In this paper we investigate parabolic Higgs bundles of varying weights and study the dependence of the harmonic metric on the weight and the Higgs bundle. The following theorem is the main result, which uses the Hitchin-Kobayashi correspondence for parabolic Higgs bundles (\cite[Thm. 6]{Simpson90}) to prove a Higgs bundle generalisation of Judge's theorem.

\begin{theorem}[(Theorem \ref{thm:analytic-parametrisation})]
For an initial stable parabolic Higgs bundle, the metric solving the self-duality equations depends analytically on the choice of weights and stable Higgs bundle in a neighbourhood of the initial weight and Higgs bundle. 
\end{theorem}

Judge's theorem then appears in the case of a fixed Higgs bundle given by a Fuchsian point in the rank $2$ moduli space (Corollary \ref{cor:Higgs-Judge}). Along the way we obtain a new Higgs bundle proof of McOwen's theorem (Corollary \ref{cor:mcowen}), thus generalising the results of Nasatyr and Steer to metrics with arbitrary cone angles.

\noindent {\bf Organisation of the paper.} In Section \ref{sec:definitions} we recall the necessary definitions and results for parabolic Higgs bundles from \cite{MehtaSeshadri80} and \cite{Simpson90} and recall the necessary results on weighted Sobolev spaces from \cite{Biquard97} that will be used in the proof of the main theorem. In Section \ref{sec:local-study} we define a model rank $2$ harmonic bundle on a punctured disk associated to a given cone angle $\theta$. The holonomy of the associated flat connection around the puncture corresponds to an elliptic element of $\SL(2, \C)$. As the cone angle converges to zero we show that the harmonic cone metrics converge to the cusp metric studied by Simpson \cite[Sec. 5]{Simpson90} and that (modulo gauge) the corresponding sequence of holonomy representations given by elliptic elements of $\SL(2, \C)$ converges to a parabolic element of $\SL(2, \C)$. Applying Simpson's nonabelian Hodge theorem \cite{Simpson90} then gives us a new proof of McOwen's theorem (Corollary \ref{cor:mcowen}).

In Section \ref{sec:convergence} we globalise the results of Section \ref{sec:local-study} and show in Theorem \ref{thm:analytic-parametrisation} that for an initial algebraically stable parabolic Higgs bundle, the harmonic metric on $E$ depends analytically on the weights and stable Higgs bundle. In particular, as a sequence of weights converges to a fixed weight, then the harmonic metrics, flat connections and holonomy representations converge. Restricting to the rank $2$ case gives a new Higgs bundle proof of Judge's theorem \cite{Judge98}.

\noindent {\bf Acknowledgements.} This project began as part of a summer REGS program funded by the GEAR network. We would also like to thank George Daskalopoulos for his advice and encouragement. The first author would like to thank the Institute for Mathematical Sciences at the National University of Singapore for their hospitality during the program ``The Geometry, Topology and Physics of Moduli Spaces of Higgs Bundles", and the second author would like to thank Brown University for their hospitality.

\section{Background and definitions}\label{sec:definitions}

\subsection{Parabolic bundles and model metrics}\label{subsec:parabolic-def}

In this section we recall some basic notions of parabolic vector bundles and Higgs bundles from \cite{MehtaSeshadri80} and \cite{Simpson90} which are relevant to the rest of the paper. Since we use the results of Simpson from \cite{Simpson90} throughout the rest of the paper, then we will also follow the terminology and notation from \cite{Simpson90} in Definitions \ref{def:filtered-Higgs}--\ref{def:analytic-degree} below.

Let $\bar{X}$ be a compact Riemann surface with marked points $\{ p_1, \ldots, p_n \}$ and let $X = \bar{X} \setminus \{p_1, \ldots, p_n \}$. Let $i : X \rightarrow \bar{X}$ denote the inclusion. We will use a Riemannian metric on $\bar{X}$ in the conformal class determined by the complex structure to define the distance from each marked point. For simplicity, in the sequel we will state the definitions for the case of one marked point $p$; the general case follows in exactly the same way.

Let $E \rightarrow X$ be a holomorphic vector bundle. The notion of a filtered bundle from \cite{Simpson90} involves a choice of extension of $E$ across the puncture $p$.

\begin{definition}[Filtered regular Higgs bundle]\label{def:filtered-Higgs}
A \emph{filtered vector bundle} is an algebraic vector bundle $E \rightarrow X$ together with a one-parameter family of vector bundles $E_\alpha \rightarrow \bar{X}$ indexed by $\alpha \in \R$ such that $E = i^* E_\alpha$ for all $\alpha$ and
\begin{itemize}

\item $E_\alpha$ is a subsheaf of $E_\beta$ for each $\alpha \geq \beta$,

\item for each $\alpha$ there exists $\varepsilon' > 0$ such that $E_{\alpha - \varepsilon} = E_\alpha$ for all $0 < \varepsilon < \varepsilon'$, and

\item $E_{\alpha+1} = E_\alpha[-p]$ for all $\alpha$.

\end{itemize}

A \emph{filtered regular Higgs bundle} $(E, \phi, \{ E_\alpha \})$ is a filtered vector bundle $(E, \{ E_\alpha \})$ together with a section $\phi \in H^0(\End(E_0) \otimes K_{\bar{X}}[p])$ such that $\phi$ preserves the subsheaf $E_\alpha \subset E_0$ for each $\alpha \in (0,1]$.
\end{definition}

The equivalence of this definition with the definition of a parabolic structure from \cite{MehtaSeshadri80} is given as follows. Given a filtered bundle $\{ E_\alpha \}_{\alpha \in \R}$, let $E_{p, 0}$ denote the fibre of $E_0 \rightarrow \bar{X}$ over $p \in \bar{X}$. Then the vector space $E_{p,0}$ has an induced filtration $\{ E_{p,\alpha} \}$ indexed by $0 \leq \alpha < 1$. For each $\alpha$, define $\Gr_\alpha(E_{p,0})$ to be the direct limit of the system $E_{p, \alpha} / E_{p, \beta}$ over all $\beta > \alpha$. The \emph{weights} of the parabolic structure are the values of $\alpha$ in $[0,1)$ such that $\dim_\C \Gr_\alpha(E_{p,0}) > 0$. In the sequel we will use $\alpha$ or $\beta$ to denote the weights of a given parabolic structure, and $\mu$ or $\nu$ to denote the set $\{ \alpha_1, \ldots, \alpha_n \}$ of weights counted with multiplicity.

In a neighbourhood $U$ of $p$ with coordinate $z$ such that $p$ corresponds to $z=0$, the Higgs field locally has the form $\varphi(z) z^{-1} dz$, where $\varphi$ is a holomorphic endomorphism of $\left. E_0 \right|_U$. The \emph{residue} of $\phi$ at $p$ is defined to be $\Res_p \phi : = \varphi(0)$. The condition that $\phi$ preserves the subsheaf $E_\alpha \subset E_0$ for each $\alpha \in (0,1]$ implies that the residue of $\phi$ respects the filtration $\{ E_{p, \alpha} \}$ defined above. In the following a \emph{parabolic Higgs bundle} will refer to the triple $((\bar{\partial}_A, \phi), \mu, \{E_{p, \alpha} \})$, where $\bar{\partial}_A$ is the holomorphic structure on $E_0$, $\phi \in H^0(\End(E_0) \otimes K_{\bar{X}}[p])$, $\mu$ is the set of weights and $\{ E_{p, \alpha} \}$ denotes the filtration on the fibre $E_p$ over the marked point.

\begin{definition}[Algebraic stability]\label{def:algebraic-stability}
Given a filtered regular Higgs bundle $(E, \phi, \{E_\alpha\})$ on $X$, the \emph{algebraic degree} is
\begin{equation}\label{eqn:algebraic-degree}
\deg(E, \phi, \{E_\alpha\}) := \deg(E_0) + \sum_{0 \leq \alpha < 1} \alpha \dim_\C(\Gr_\alpha(E_{p,0})) .
\end{equation}
The filtered regular Higgs bundle $(E, \phi, \{E_\alpha\})$ is \emph{algebraically stable} (resp. \emph{semistable}) iff for all filtered regular Higgs subbundles $(F, \phi, \{F_\alpha\}) \subset (E, \phi, \{E_\alpha\})$ we have
\begin{equation}\label{eqn:algebraic-slope-stability}
\frac{\deg(F, \phi, \{F_\alpha\})}{\rank(F)} < \frac{\deg(E, \phi, \{E_\alpha\})}{\rank(E)} \quad \text{(resp. $\leq$)} .
\end{equation}
\end{definition}

A filtered regular Higgs bundle together with the notion of algebraic stability is a purely algebraic object. In order to relate these objects to flat connections and representations of $\pi_1(X)$, we need to use a Hermitian metric on the bundle. For a noncompact surface, this requires the imposition of growth conditions at the marked points (cf. \cite[Sec. 3]{Simpson90}).
\begin{definition}[Acceptable metric]\label{def:acceptable-metric}
Let $E \rightarrow X$ be a holomorphic bundle with a smooth Hermitian metric $h$, and let $F_h$ denote the curvature of the Chern connection. Let $U \subset X$ be a neighbourhood of the marked point $p \in \bar{X}$ with coordinate $r$ denoting the distance from $p$. The metric $h$ on $E$ is \emph{acceptable} if $| F_h |_h \leq f + \frac{1}{r^2 (\log r)^2}$ for some $f \in L^q$ with $q > 1$.
\end{definition}
In \cite[Sec. 10]{Simpson88} (see also \cite[Prop. 3.1]{Simpson90}), Simpson proves that if the metric $h$ is acceptable then there is a filtered bundle $(E, \phi, \{ E_\alpha \})$ associated to $(E, \phi, h)$, where the germs of sections of $E_\alpha$ at the puncture $p$ are local sections $s$ of $E$ satisfying the growth condition 
\begin{equation}\label{eqn:growth-condition}
|s|_h \leq C r^{\alpha - \varepsilon}
\end{equation}
for all $\varepsilon > 0$. The acceptability of the metric guarantees that the sheaves $E_\alpha$ will be coherent. Moreover, Simpson also shows in \cite[Thm. 2]{Simpson90} that if the harmonic bundle is \emph{tame} (the growth of the eigenvalues of the Higgs field is bounded by a constant times $\frac{1}{r}$ in a neighbourhood of the puncture) then the harmonic metric is acceptable.

\begin{definition}[Analytic degree]\label{def:analytic-degree}
Given a holomorphic bundle $E \rightarrow X$ with an acceptable Hermitian metric $h$, let $F_h$ denote the curvature of the Chern connection. The \emph{analytic degree} is
\begin{equation*}
\deg(E, h) = \int_{X} \tr(F_h) .
\end{equation*}
\end{definition}

Simpson proves in \cite[Lem. 6.1]{Simpson90} that the analytic and algebraic degree are equal if the metric $h$ is acceptable. With respect to the analytic degree, one can define analytic slope stability in an analogous way to \eqref{eqn:algebraic-slope-stability}. Simpson proves in \cite[Lem. 6.3]{Simpson90} that the two definitions of stability are equivalent if the metric is acceptable.

Now we describe the model asymptotic behaviour of the metrics near the marked point $p$, following \cite[Sec. 2]{Biquard97} and \cite[Sec. 7]{Simpson90}. Since $\Res_p \phi$ preserves the filtration $\{ E_{p, \alpha} \}$ then the graded pieces decompose as a direct sum $\Gr_\alpha(E_{p,0}) = \bigoplus_\lambda \Gr_\alpha^\lambda(E_{p,0})$ according to the generalised eigenspaces of $\Res_p \phi$, and the residue induces a nilpotent endomorphism $Y_\alpha$ on each $\Gr_\alpha(E_{p,0})$ by taking the upper triangular part of each Jordan block. The $Y_\alpha$ then induces a further filtration $\{ W_k \Gr_\alpha^\lambda(E_{p,0}) \}_{k \in \mathbb{Z}}$ called the \emph{weight filtration}, with corresponding grading 
\begin{equation}\label{eqn:weight-grading}
\Gr_\alpha(E_{p,0}) = \bigoplus_{k \in \mathbb{Z}} \bigoplus_\lambda \Gr_k \Gr_\alpha^\lambda(E_{p,0}) 
\end{equation}
such that $Y_\alpha ( \Gr_k \Gr_\alpha^\lambda) \subset \Gr_{k-2} \Gr_\alpha^\lambda$. Therefore if we define the diagonal endomorphism $H_\alpha = \bigoplus_{k \in \mathbb{Z}} k \cdot \id_{\Gr_k \Gr_\alpha^\lambda}$ then $[H_\alpha, Y_\alpha] = -2 Y_\alpha$. Then there exists an endomorphism $X_\alpha$ such that $(H_\alpha, X_\alpha, Y_\alpha)$ are the generators of a representation of $\mathfrak{sl}_2$ on $\Gr_\alpha^\lambda(E_{p,0})$, i.e. we also have $[H_\alpha, X_\alpha] = 2X_\alpha$ and $[X_\alpha, Y_\alpha] = H_\alpha$. 

Now choose an initial metric $h_p$ on $E_{p,0}$ such that the subspaces $\Gr_\alpha(E_{p,0})$ are orthogonal and such that $H_\alpha^* = H_\alpha$ and $Y_\alpha^* = X_\alpha$. Given a trivialisation of $E_0 \rightarrow \bar{X}$ in a neighbourhood $U$ of $p$ with a projection $\pi : U \rightarrow \{p\}$, we can pullback by $\pi$ to extend the weight filtration, the grading \eqref{eqn:weight-grading} and the $\mathfrak{sl}_2$ representation to this neighbourhood. Let $r$ denote the distance to the marked point $p$ in the neighbourhood $U$.

\begin{definition}[Model metric near a marked point]\label{def:model-metric}
The \emph{model metric} on $\left. E \right|_{U \setminus \{p\}}$ is defined with respect to the grading \eqref{eqn:weight-grading} by
\begin{equation}\label{eqn:model-metric}
h_{mod} = \bigoplus_{k, \alpha, \lambda} r^{2 \alpha} \left| \log r \right|^k \cdot \left( \left. \pi^* h_p \right|_{\Gr_k \Gr_\alpha^\lambda} \right) 
\end{equation}
\end{definition}

\begin{example}
The two basic building blocks for the weight filtration are described by Simpson in \cite[Sec. 5]{Simpson90}. The first is where $\dim_\C \Gr_\alpha^\lambda = 1$ and $k = 0$ from \cite[p745]{Simpson90}. The second is where $\dim_\C \Gr_\alpha^\lambda = 2$, the weight is $\alpha = 0$ and the residue of $\phi$ is nilpotent, therefore the eigenvalue is $\lambda = 0$. This determines a decomposition $E_{p,0} = E_{p,0}^{1,0} \oplus E_{p,0}^{0,1}$ and a nilpotent endomorphism $Y : E_{p,0}^{1,0} \rightarrow E_{p,0}^{0,1}$. Then the graded object of the weight filtration has two pieces $E_{p,0}^{1,0}$ (corresponding to $k=1$) and $E_{p,0}^{0,1}$ (corresponding to $k=-1$) so the model metric has the form $h_{mod} = |\log r|$ on $E_{p,0}^{1,0}$ and $h_{mod} = |\log r|^{-1}$ on $E_{p,0}^{0,1}$. The general case described above is given by taking symmetric powers and tensor products of these two basic examples.
\end{example}

\begin{definition}[Bounded distance between metrics]
Given two metrics $h, k$ on $E \rightarrow X$, we say that \emph{$h$ is bounded with respect to $k$} if $\sup_X \dist_{G/K}(h, k) < \infty$, where $\dist_{G/K}$ refers to the geodesic distance in the symmetric space $G/K$.
\end{definition}

\begin{remark}
It will be useful in the next section to note that for $r \in \left(0,\frac{1}{2} \right)$, the metric $h(r) = - \frac{1}{\theta} \sinh(\theta \log r) = \frac{1}{2 \theta} \left( r^{-\theta} - r^\theta \right)$ with values in $\mathbb{R}_{>0} \cong \GL(1,\C) / \U(1)$ is a bounded distance from the metric $r^{-\theta}$.
\end{remark}

The following is the nonabelian Hodge theorem for parabolic Higgs bundles on Riemann surfaces from \cite[Thm. 6]{Simpson90} (see also \cite[Thm. 8.1]{Biquard97} for the higher dimensional case).

\begin{theorem}\label{thm:nonabelian-hodge}
Let $\bar{X}$ be a compact Riemann surface and let $X = \bar{X} \setminus \{p_1, \ldots, p_m \}$. Let $(E, \phi, \{E_\alpha\})$ be an algebraically stable filtered regular Higgs bundle. Then there exists a Hermitian-Einstein metric $h$ on $E$ such that in a neighbourhood $U_j$ of each marked point $p_j$ there exists a finite $C_j > 0$ such that $\sup_{U_j} \dist(h, h_{mod}) < C_j$.
\end{theorem}


\subsection{Weighted Sobolev spaces}\label{subsec:weighted-sobolev}

This section contains the definitions of weighted Sobolev and H\"older spaces that are used in the proof of Theorem \ref{thm:analytic-parametrisation}. We follow closely the notation and setup of \cite{Biquard97}; other useful references are \cite{Adams75} and \cite{Biquard91}.


In the previous section we fixed a complex structure on the compact surface $\bar{X}$, which induces a complex structure on the punctured surface $X = \bar{X} \setminus \{p_1, \ldots, p_n\}$ and therefore a conformal class of Riemannian metrics on $X$. Within this class we choose a complete metric on $X$ with cusp singularities near the marked points (cf. \cite[Sec. 2C]{Biquard97}).

Given a filtered regular Higgs bundle $(E, \phi, \{ E_\alpha \})$ with a set of weights $\mu = \{ \alpha_1, \ldots, \alpha_n \}$, and a metric $h_\mu$ equal to the model metric \eqref{eqn:model-metric} near each marked point, let $d_{h_\mu} = \bar{\partial} + \partial_{h_\mu}$ denote the Chern connection with respect to the metric $h_\mu$ and the holomorphic structure $\bar{\partial}$ on $E$, and let $D_{h_\mu} = d_{h_\mu} + \phi + \phi^*$ denote the associated $\GL(n, \C)$ connection. Following Simpson's notation (cf. \cite[p13]{Simpson92}), we also define the operators $D'' = \bar{\partial} + \phi$ and $D_{h_\mu}' = \partial_{h_\mu} + \phi^* = D_{h_\mu} - D''$.


As in the previous section, in a neighbourhood of each marked point $p_j$ let $r$ denote the distance from $p_j$ with respect to the metric on $\bar{X}$. In a neighbourhood $U$ of each marked point, let $y = |\log r|$. Then the coordinates $(y, \theta)$ define an infinite cylinder on which the \emph{weighted $L^p$ norm with weight $\delta$} of a section $\eta \in \Omega^\ell(U, E)$ is (cf. \cite{LockhartMcOwen85})
\begin{equation*}
\| \eta \|_{L_\delta^p} := \left( \int_U \left| y^\delta \eta \right|^p \frac{1}{y} dy d \theta \right)^{\frac{1}{p}} .
\end{equation*}
Let $t$ be a smooth function on $X$ equal to $\log \left| \log r \right| = \log y$ near $\{ p_1, \ldots, p_n \}$ and equal to zero far from this set. As noted in \cite[Sec. 4A]{Biquard97}, in a neighbourhood $U$ of the marked points the above norm is equivalent to 
\begin{equation}\label{eqn:cylinder-weight-space}
\left( \int_U |e^{\delta t} \eta |^p \, dt d \theta \right)^{\frac{1}{p}} .
\end{equation}
Using the function $t$, we can extend the above norm to a norm $\| \cdot \|_{L_\delta^p}$ on all of $X$. 

Let $\nabla$ denote the covariant derivative associated to the connection $D_\mu$. The \emph{weighted Sobolev norms} are defined for $\eta \in \Omega^\ell(X, \End(E))$ by
\begin{align*}
\| \eta \|_{L_\delta^{k,p}} := \sum_{j=0}^k \| \nabla^j \eta \|_{L_\delta^{p}} .
\end{align*}
Define the space $L_\delta^{k,p}(\Omega^\ell(X, \End(E)))$ as the completion of $\Omega^\ell(X, \End(E))$ in this norm. Let $\chi(t)$ be a smooth increasing function such that $\chi(0) = 0$ and $\chi(t) = 1$ when $t$ is large and define
\begin{equation}\label{eqn:L-hat-weighted-Sobolev}
\hat{L}_\delta^{k,p}(\Omega^\ell(X, \End(E))) := \left\{ \eta = C \chi(t) \cdot \id + \eta_1 \, : \, \text{$C$ is constant and $\eta_1 \in L_\delta^{k,p}$} \right\} .
\end{equation}
Note that since we only consider discrete marked points on a compact Riemann surface then this is a slight simplification of Biquard's definition from \cite[(4.3), (4.4)]{Biquard97} for a smooth divisor on a compact K\"ahler manifold. From now on we drop the notation for the bundle and the degree of the differential form, and use $\hat{L}_\delta^{k,p} := \hat{L}_\delta^{k,p}(\Omega^0(X, \End(E)))$, unless it is necessary to include the extra notation. Since the constant scalar multiples of the identity are in the kernel of any connection $D$, then this connection is well-defined on $\tilde{L}_\delta^{k,p} := \hat{L}_\delta^{k,p}(\Omega^0(X, \End(E))) / \C$, the quotient by the constant scalar multiples of the identity.

From the multiplication and embedding theorems for weighted Sobolev spaces we have the following result from \cite[Lem. 4.6]{Biquard97} which will be used in the sequel. 
\begin{lemma}\label{lem:sobolev-algebra}
Let $n = \dim_\C X$. If $k - \frac{2n}{p} > 0$ then $\hat{L}_\delta^{k,p}$ is an algebra and for all $j \leq k$ the space $\hat{L}_\delta^{j,p}$ is an $\hat{L}_\delta^{k,p}$-module.
\end{lemma}

Now let $P$ be the $GL(n, \C)$ principal bundle associated to $E \rightarrow X$, and let $\Ad(P)$ be the associated adjoint bundle. The gauge group is 
\begin{equation}\label{eqn:weighted-gauge-group}
( \mathcal{G}^\C )_\delta^{k,p} := \{ \text{$g$ takes values in $\Ad(P)$} \, : \, (\nabla g) g^{-1}  \in \hat{L}_\delta^{k-1,p} \} 
\end{equation}
and the space of metrics is
\begin{equation}\label{eqn:weighted-space-metrics}
\mathcal{H}_\delta^{k,p} := \{ g^* g \, : \, g \in (\mathcal{G}^\C)_\delta^{k,p} \} .
\end{equation}

As above, given a filtered regular Higgs bundle $(E, \phi, \{ E_\alpha \})$ and a metric $h_\mu$ equal to the model metric \eqref{eqn:model-metric} near each marked point, define the connection $D_{h_\mu} = d_{h_\mu} + \phi + \phi^*$, where $d_{h_\mu}$ is the Chern connection of $E$ with respect to the metric $h_\mu$. Now define the space of all connections as
\begin{equation*}
\mathcal{A}_\delta^{1,p} := \{ D_{h_\mu} + a \, : \, a \in \hat{L}_\delta^{1,p}(\Omega^1(\End(E)) \} .
\end{equation*} 
If $p > 2n$ then Lemma \ref{lem:sobolev-algebra} implies that $(\mathcal{G}^\C )_\delta^{2,p}$ acts continuously on $\mathcal{A}_\delta^{1,p}$ and that the curvature of any connection $D \in \mathcal{A}_\delta^{1,p}$ satisfies $F_D \in \hat{L}_\delta^p$. In the following we choose $p > 2n$ so that Lemma \ref{lem:sobolev-algebra} applies for all $k \geq 1$, we choose $\delta > 0$ small enough so that Proposition \ref{prop:Fredholm-index-zero} below holds, and drop the notation for $k$, $p$ and $\delta$ from the gauge group and space of metrics. 

The following two results will be used in the proof of Theorem \ref{thm:analytic-parametrisation}.

\begin{lemma}\label{lem:perturbation-bounded}
Fix a Higgs bundle $(E, \phi) \rightarrow X$ and let $\mu$, $\nu$ be two different sets of weights such that the associated parabolic Higgs bundles are algebraically stable with algebraic degree zero. Let $h_{\mu}$ be a Hermitian-Einstein metric with weight $\mu$, let $h_{\nu}$ be a model metric with respect to the set of weights $\nu$, define $k := h_{\mu}^{-1} h_{\nu}$ and choose a section $g$ of $\Ad(P)$ such that $k = g^* g$. Then for $0 < \delta < \frac{1}{2}$ we have $g \Lambda D'' (k^{-1} D_{h_{\mu}}'  k) g^{-1} \in \hat{L}_\delta^p$. 
\end{lemma}

\begin{proof}
The curvatures are related by the formula
\begin{equation}\label{eqn:curvature-relation}
g^{-1} \Lambda(F_{h_{\nu}} + [\phi, \phi^*]) g - \Lambda( F_{h_{\mu}} + [\phi, \phi^*]) =  \Lambda D'' (k^{-1} D_{h_{\mu}}' k) .
\end{equation}
Since $h_\mu$ is Hermitian-Einstein and the parabolic degree is zero then the second term on the left-hand side vanishes. Biquard \cite[(3.4)]{Biquard97} shows that the model metric satisfies 
\begin{equation*}
F_{h_{\nu}} + [\phi, \phi^*] = const + O \left( | \log r |^{-\frac{1}{2}} \right)
\end{equation*}
(again this is a simplification of Biquard's results to the case where the marked points are discrete). Replacing $t = \log |\log r|$ gives us
\begin{equation*}
F_{h_{\nu}} + [\phi, \phi^*] = const + O \left( e^{-\frac{1}{2} t} \right)
\end{equation*}
and so \eqref{eqn:cylinder-weight-space} shows that $\Lambda (F_{h_{\nu}} + [\phi, \phi^*]) \in \hat{L}_\delta^p$ if $0 < \delta < \frac{1}{2}$. Therefore \eqref{eqn:curvature-relation} shows that $g \Lambda D'' (k^{-1} D_{h_{\mu}}'  k) g^{-1} \in \hat{L}_\delta^p$.
\end{proof}

\begin{lemma}\label{lem:perturbed-metric}
Fix a connection $D_{h_\mu}$ associated to a filtered regular Higgs bundle $(E, \phi, \{ E_\alpha \})$ as above, and define $D'' = \bar{\partial} + \phi$, $D_{h_\mu}' = \partial_{h_\mu} + \phi^*$. Given any $g \in \mathcal{G}^\C$ , let $k = g^* g$. Then $D'' \left( k^{-1} D_{h_\mu}' k \right) \in \hat{L}_\delta^p$.
\end{lemma}

\begin{proof}
Since $g \in \mathcal{G}^\C$, then $(\nabla g) g^{-1} \in \hat{L}_\delta^{1,p}$ by definition. Therefore Lemma \ref{lem:sobolev-algebra} shows that $k^{-1} D_{h_\mu}' k$ is also in $\hat{L}_\delta^{1,p}$, and so $D'' \left( k^{-1} D_{h_\mu}' k \right) \in \hat{L}_\delta^p$.
\end{proof}

The following restatement of a theorem of Biquard \cite[Thm. 5.1]{Biquard97} will be used in the proof of Theorem \ref{thm:analytic-parametrisation}.
\begin{proposition}\label{prop:Fredholm-index-zero}
Let $n = \dim_\C X$, let $(E, \phi, \{ E_\alpha \})$ be a filtered regular Higgs bundle with set of weights $\mu$, let $h$ be a metric within a bounded distance of the model metric associated to $(E, \phi, \{ E_\alpha \})$ and let $D_{h_\mu}$ be the associated connection with curvature $F_{h_\mu} + [\phi, \phi^*]$. If $\delta > 0$ is small enough and if $p > 2n$, then the Laplacian
\begin{equation}
D_{h_\mu}^* D_{h_\mu} : \hat{L}_\delta^{2,p}(\Omega^0(\End(E))) \rightarrow \hat{L}_\delta^p(\Omega^0(\End(E)))
\end{equation}
is Fredholm of index zero. If $F_{h_\mu} + [\phi, \phi^*] = 0$, then the same is true for $(D_{h_\mu}')^* D_{h_\mu}'$ and $(D'')^* D''$. If the restriction of $D''$ to $\tilde{L}_\delta^{2,p}$ is injective then the restrictions $(D_{h_\mu}')^* D_{h_\mu}' : \tilde{L}_\delta^{2,p} \rightarrow \tilde{L}_\delta^{0,p}$ and $(D'')^* D'' : \tilde{L}_\delta^{2,p} \rightarrow \tilde{L}_\delta^{0,p}$ are both injective and surjective.
\end{proposition}

\begin{proof}
The first statement follows directly from \cite[Thm. 5.1]{Biquard97}. Note that since the divisor $D = p_1 + \cdots + p_n$ consists of isolated points when $\dim_\C X = 1$, then the condition $\Delta_D^A(\left. f \right|_D) = 0$ from \cite[Thm. 5.1]{Biquard97} is trivial in our case.

It follows from the K\"ahler identities that $i \Lambda \left( F_{h_\mu} + [\phi, \phi^*] \right) = (D_{h_\mu}')^* D_{h_\mu}' - (D'')^* D''$ (cf. \cite{Simpson92}). Therefore $F_{h_\mu} + [\phi, \phi^*] = 0$ implies that $(D'')^* D'' = (D_{h_\mu}')^* D_{h_\mu}' = \frac{1}{2} D_{h_\mu}^* D_{h_\mu}$, and so the operators 
\begin{align*}
(D_{h_\mu}')^* D_{h_\mu}' : \hat{L}_\delta^{2,p}(\Omega^0(\End(E))) & \rightarrow \hat{L}_\delta^p(\Omega^0(\End(E))) \\
(D'')^* D'' : \hat{L}_\delta^{2,p}(\Omega^0(\End(E))) & \rightarrow \hat{L}_\delta^p(\Omega^0(\End(E)))
\end{align*}
are also Fredholm of index zero. Moreover, if the restriction of $D''$ to $\tilde{L}_\delta^{2,p}$ is injective, then the same is true for $D_{h_\mu}'$, and therefore since $(D'')^* D''$ and $(D_{h_\mu}')^* D_{h_\mu}'$ have index zero then they are also surjective onto $\tilde{L}_\delta^p$. Finally, if $u \in \tilde{L}_\delta^{2,p}$ is self-adjoint with respect to the metric, then
\begin{equation*}
\left( (D_{h_\mu}')^* D_{h_\mu}' u \right)^* = (D'')^* D'' u = (D_{h_\mu}')^* D_{h_\mu}' u .
\end{equation*}
Conversely, if $(D_{h_\mu}')^* D_{h_\mu}' u$ is self-adjoint, then 
\begin{equation*}
(D_{h_\mu}')^* D_{h_\mu}' u = \left( (D_{h_\mu}')^* D_{h_\mu}' u \right)^* = (D'')^* D'' u^* = (D_{h_\mu}')^* D_{h_\mu}' u^* 
\end{equation*}
and so $u - u^* \in \ker D_{h_\mu}'$ which implies that $u$ must be self-adjoint since $D_{h_\mu}'$ is injective by assumption. Therefore $(D_{h_\mu}')^* D_{h_\mu}'$ maps the self-adjoint sections of $\tilde{L}_\delta^{2,p}$ surjectively onto the self-adjoint sections of $\tilde{L}_\delta^p$.
\end{proof}

\begin{remark}
Here we use Biquard's weighted Sobolev spaces in order to use \cite[Thm. 5.1]{Biquard97} to prove Proposition \ref{prop:Fredholm-index-zero}, which applies in full generality. In the case of a Higgs bundle at a Fuchsian point in the moduli space, where the Hermitian metric on the bundle determines a hyperbolic metric on the punctured surface (cf. Corollary \ref{cor:mcowen} and Corollary \ref{cor:Higgs-Judge}), Judge \cite{Judge98} develops the Fredholm theory using a variant of H\"older spaces, which differ from Biquard's H\"older spaces in \cite[Sec. 4A]{Biquard97}. One could also try to extend Judge's construction to higher rank Higgs bundles, but we avoid this approach here since Biquard's theory is already available.
\end{remark}

\section{Local study}\label{sec:local-study}

In this section we explicitly describe the nonabelian Hodge correspondence for a fixed Higgs bundle on the punctured unit disk $\mathbb{D}_0 := \mathbb{D} \setminus \{0\}$ with varying weights and prove Proposition \ref{prop:rank-2-local-study}, which shows that the harmonic bundles with cone angle $\theta$ converge as $\theta \rightarrow 0$ to the harmonic bundle with a cusp metric studied by Simpson in \cite{Simpson90}. This is a local version of the main theorem of the next section.

The proof is by explicit calculation for the case $G = \SL(2, \C)$. In the following we fix a filtered regular Higgs bundle $(E, \phi, \{ E_\alpha \})$ on the punctured disk $\mathbb{D}_0$ and a Hermitian-Einstein metric $h$ on $E$. From the triple $(E, \phi, h)$ one can construct a flat connection $D$ on $E$ which has an associated holonomy representation $\rho : \mathbb{Z} \rightarrow G$. The Hermitian-Einstein metric then determines a $\mathbb{Z}$-equivariant harmonic map $h : \tilde{\mathbb{D}}_0 \rightarrow G/K$, where $\mathbb{Z} = \pi_1(\mathbb{D}_0)$ acts on the universal cover $\tilde{\mathbb{D}}_0$ by deck transformations and on $G / K$ via the holonomy representation $\rho$.



Now we study in more detail the sequence of harmonic bundles corresponding to hyperbolic cone metrics. Let $\mathbb{D}_0$ denote the punctured unit disk, and choose a branch of log
\begin{equation}\label{eqn:branch-log}
U = \left\{ z = re^{i \gamma} \in \mathbb{D}_0 \, : \, \gamma \in (-\pi, \pi) \right\} .
\end{equation} 
Let $E \rightarrow \mathbb{D}_0$ be a rank $2$ complex vector bundle with a trivialisation over $U$. Define a Higgs structure on $E$ by taking the trivial holomorphic structure and defining the Higgs field on the trivialisation over $U$ by
\begin{equation*}
\phi(z) = \left( \begin{matrix} 0 & 0 \\ \frac{1}{2} & 0 \end{matrix} \right)  z^{-1} dz
\end{equation*}
Let $w^{1,0}$ and $w^{0,1}$ be a basis for the holomorphic sections of $E$ such that
\begin{equation*}
\phi(z) w^{1,0} = \frac{1}{2} w^{0,1} z^{-1} dz, \quad \phi(z) w^{0,1} = 0 .
\end{equation*}
Let $E \cong E^{1,0} \oplus E^{0,1}$ be the direct sum decomposition with respect to these sections. Define the Hermitian metric
\begin{equation}\label{eqn:cone-metric}
k_\theta(r) = \left( \begin{matrix} \frac{1}{2\theta} (r^{-\theta} - r^{\theta}) & 0 \\ 0 & \frac{2\theta}{r^{-\theta} - r^{\theta}} \end{matrix} \right) = \left( \begin{matrix} - \frac{1}{\theta} \sinh (\theta \log r) & 0 \\ 0 & -\frac{1}{\frac{1}{\theta} \sinh (\theta \log r)} \end{matrix} \right)
\end{equation}
With respect to this metric, we have
\begin{equation}\label{eqn:growth-rate}
| w^{1,0} |_{k_\theta} = \frac{r^{- \frac{1}{2} \theta} }{\sqrt{2 \theta}} (1 - r^{2\theta})^\frac{1}{2} = O(r^{-\frac{1}{2} \theta}), \quad |w^{0,1} |_{k_\theta} = \frac{\sqrt{2 \theta} r^{\frac{1}{2}\theta}}{(1 - r^{2 \theta})^{\frac{1}{2}}} = O(r^{\frac{1}{2}\theta}) 
\end{equation}
and so the weights in the interval $[0,1)$ are $\frac{1}{2}\theta$ and $1 - \frac{1}{2} \theta$. Equation \eqref{eqn:model-curvature} below shows that $k_\theta$ is an acceptable metric and so the result of \cite[Prop. 3.1]{Simpson90} shows that there is an associated filtered bundle which is determined using the growth rate of the sections from equation \eqref{eqn:growth-rate} above. From \eqref{eqn:growth-condition} we see that the extension of the bundle $E$ across the puncture at weight zero is $\mathcal{O}(E^{1,0})[-p] \oplus \mathcal{O}(E^{0,1})$. At weight $\frac{1}{2}\theta$ the extension of the bundle $E^{0,1}$ across the puncture jumps from $\mathcal{O}(E^{0,1})$ to $\mathcal{O}(E^{0,1})[-p]$ and at weight $1- \frac{1}{2} \theta$ the extension of $E^{1,0}$ across the puncture jumps from $\mathcal{O}(E^{1,0})[-p]$ to $\mathcal{O}(E^{1,0})[-2p]$.

In \cite[p746]{Simpson90}, Simpson proves that the metric
\begin{equation}\label{eqn:Simpson-rank2-local-model}
k_0(r) = \left( \begin{matrix} -\log r & 0 \\ 0 & -\frac{1}{\log r} \end{matrix} \right)
\end{equation}
is Hermitian-Einstein with respect to the Higgs bundle $(\bar{\partial}, \phi)$ on the trivialisation $U$ and writes down a basis for the flat sections $v^{0,1}$ and $u^{1,0}$  of the associated flat connection $D_0$. Explicitly, the flat connection is given by
\begin{equation*}
D_0 = d + \left( \begin{matrix} \frac{1}{2} \log r & 0 \\ \frac{1}{2} & - \frac{1}{2} \log r \end{matrix} \right) z^{-1} dz + \left( \begin{matrix} 0 & \frac{1}{2 (\log r)^2} \\ 0 & 0 \end{matrix} \right) \bar{z}^{-1} d\bar{z} 
\end{equation*}
and the flat sections by $v^{0,1} = w^{0,1} + \frac{w^{1,0}}{\log r}$ and $u^{1,0} = w^{1,0} - \frac{1}{2} \log z \, w^{0,1}$. With respect to this basis for the flat sections, the associated representation $\rho_0 : \mathbb{Z} \rightarrow \SL(2, \C)$ maps the generator of $\mathbb{Z}$ to $\left( \begin{matrix} 1 & \pi \\ 0 & 1 \end{matrix} \right)$.

The next result shows that the harmonic bundle and holonomy representation with cone angle $\theta$ converge to the harmonic bundle and holonomy representation with cone angle zero studied by Simpson.

\begin{proposition}\label{prop:rank-2-local-study}
The metric $k_\theta$ is Hermitian-Einstein with respect to the Higgs bundle $(\bar{\partial}, \phi)$ and $k_\theta$ depends analytically on $\theta$. The monodromy representation $\rho_\theta : \mathbb{Z} \rightarrow \SL(2, \C)$ converges to the representation $\rho_0 : \mathbb{Z} \rightarrow \SL(2, \C)$. 
\end{proposition}

\begin{proof}
Given the metric $k_\theta$, the metric connection is given by $d_\theta = \bar{\partial} + \partial_\theta$, where
\begin{equation*}
\partial_\theta = \partial + \frac{1}{2} \theta \cotanh (\theta \log r) \left( \begin{matrix} 1 & 0 \\ 0 & -1 \end{matrix} \right) z^{-1} dz
\end{equation*}
The Hermitian adjoint of the Higgs field with respect to $k_\theta$ is
\begin{equation*}
\phi^{*_\theta}(z) = \frac{\theta^2}{\sinh^2(\theta \log r)} \left( \begin{matrix} 0 & \frac{1}{2} \\ 0 & 0 \end{matrix} \right) \bar{z}^{-1} d\bar{z}
\end{equation*}
A calculation shows that the curvature of $k_\theta$ is 
\begin{equation}\label{eqn:model-curvature}
F_{k_\theta} = \bar{\partial} (k_\theta^{-1} \partial k_\theta) = -\frac{\theta^2}{4 r^2 \sinh^2(\theta \log r)} \left( \begin{matrix} 1 & 0 \\ 0 & -1 \end{matrix} \right) d \bar{z} dz
\end{equation}
Therefore $|F_{k_\theta}| \leq \frac{1}{r^2 (\log r)^2}$ in a neighbourhood of $r=0$ and so the curvature is acceptable in the sense of Definition \ref{def:acceptable-metric}. We also have
\begin{equation*}
F_{k_\theta} + [\phi, \phi^{*_\theta}] = -\frac{\theta^2}{4 r^2 \sinh^2(\theta \log r)} \left( \begin{matrix} 1 & 0 \\ 0 & -1 \end{matrix} \right) d \bar{z} dz + \frac{\theta^2}{4 r^2 \sinh^2(\theta \log r)} \left( \begin{matrix} 1 & 0 \\ 0 & -1 \end{matrix} \right) d \bar{z} dz = 0 .
\end{equation*}
Therefore the metric is Hermitian-Einstein for all $\theta$ and the connection $D_\theta = \bar{\partial} + \partial_\theta + \phi + \phi^{*_\theta}$ is flat. As $\theta \rightarrow 0$, we have $\frac{1}{2\theta} (r^{-\theta}-r^{\theta}) = - \frac{1}{\theta} \sinh(\theta \log r) \rightarrow -\log r$ and, so the metric $k_\theta$ converges to $k_0$ and the flat connection $D_\theta$ converges to $D_0$.

We can compute the basis elements for the flat sections as follows. Let $w^{0,1}$ and $w^{1,0}$ be a basis for the holomorphic sections of $E$ in the trivialisation over $U$ such that $\phi w^{1,0} = \frac{1}{2} z^{-1} dz w^{0,1}$ and $\phi w^{0,1} = 0$. Then $d_\theta'' := \bar{\partial} + \phi^{*_\theta}$ has holomorphic sections given by $w^{1,0}$ and $v_\theta^{0,1} = w^{0,1} + \theta \cotanh(\theta \log r) w^{1,0}$. Note that $v_\theta^{0,1} \rightarrow w^{0,1} + \frac{1}{\log r} w^{1,0} =: v_0^{0,1}$ as $\theta \rightarrow 0$. A calculation shows that 
\begin{equation*}
d_\theta' w^{1,0} = \frac{1}{2} v_\theta^{0,1} z^{-1} dz, \quad d_\theta' v_\theta^{0,1} = \frac{1}{2} \theta^2 w^{1,0} z^{-1} dz
\end{equation*}
Therefore the sections 
\begin{equation*}
s_1 = z^{-\frac{\theta}{2}} (\theta w^{1,0} + v_\theta^{0,1}), \quad s_2 = z^{\frac{\theta}{2}} (\theta w^{1,0} - v_\theta^{0,1})
\end{equation*}
are flat with respect to $D_\theta = d_\theta'' + d_\theta'$. In this basis it is clear that the parallel transport along a loop around the puncture is given by $(s_1, s_2) \mapsto (e^{-\pi i \theta} s_1, e^{\pi i \theta} s_2)$. To see that this converges to the local model given in Simpson in \cite[p746]{Simpson90}, it is more convenient to apply a gauge transformation and use the basis of flat sections given by
\begin{align*}
u^{1,0} = \frac{1}{2\theta} s_1 + \frac{1}{2\theta} s_2 & = \frac{1}{2} (z^{-\frac{\theta}{2}} + z^{\frac{\theta}{2}}) w^{1,0} + \frac{1}{2\theta} (z^{-\frac{\theta}{2}} - z^{\frac{\theta}{2}}) v_\theta^{0,1} \\
u^{0,1} = \frac{1}{2} s_1 - \frac{1}{2} s_2 & = \frac{1}{2} \theta(z^{-\frac{\theta}{2}} - z^{\frac{\theta}{2}}) w^{1,0} + \frac{1}{2} (z^{-\frac{\theta}{2}} + z^{\frac{\theta}{2}}) v_\theta^{0,1}
\end{align*}
As $\theta \rightarrow 0$, these sections converge to $u^{0,1} = v_0^{0,1} = w^{0,1} + \frac{1}{\log r} w^{1,0}$ and $u^{1,0} = w^{1,0} - \frac{1}{2} v_0^{0,1} \log z$, which are the flat sections for the local model from \cite[p746]{Simpson90}. Therefore the representations $\rho_\theta : \mathbb{Z} \rightarrow \SL(2, \C)$, which map a generator of $\mathbb{Z}$ to the elliptic element $\left( \begin{matrix} e^{-\pi i \theta} & 0 \\ 0 & e^{\pi i \theta} \end{matrix} \right) \in\SL(2, \C)$ in the basis defined by $s_1$ and $s_2$, converge (after changing to the basis defined by the sections $u^{1,0}$ and $u^{0,1}$) to the representation $\rho_0 : \mathbb{Z} \rightarrow \SL(2, \C)$ from Simpson's local model \cite[p746]{Simpson90}, which maps a generator of $\mathbb{Z}$ to the parabolic element $\left( \begin{matrix} 1 & \pi \\ 0 & 1 \end{matrix} \right) \in \SL(2, \C)$.

For this local model on the punctured disk, one can see explicitly that the Hermitian-Einstein metric
\begin{equation*}
k_\theta(r) = \left( \begin{matrix} \frac{1}{2\theta} (r^{-\theta} - r^{\theta}) & 0 \\ 0 & \frac{2\theta}{r^{-\theta} - r^{\theta}} \end{matrix} \right) = \left( \begin{matrix} - \frac{1}{\theta} \sinh (\theta \log r) & 0 \\ 0 & -\frac{1}{\frac{1}{\theta} \sinh (\theta \log r)} \end{matrix} \right)
\end{equation*}
depends analytically on the weight $\frac{1}{2}\theta$. 
\end{proof}

\begin{remark}
In Section \ref{sec:convergence} we will generalise this statement to show that the Hermitian-Einstein metric on the punctured surface depends analytically on the parabolic weights.
\end{remark}

Proposition \ref{prop:rank-2-local-study} shows that the model harmonic bundle with cone angle $\theta$ and elliptic monodromy converges to Simpson's local model \cite[p746]{Simpson90} with parabolic monodromy. To set up the proof of Corollary \ref{cor:mcowen}, we now summarise analogous calculations for a gauge-equivalent local model. Given $\beta \in \R$, define the following gauge transformation on the bundle $E \rightarrow U$ from \eqref{eqn:branch-log}
\begin{equation*}
g_\beta(z) := \left( \begin{matrix} z^{-\frac{1}{2} \beta} & 0 \\ 0 & z^{\frac{1}{2} \beta} \end{matrix} \right) .
\end{equation*}
Since $g_\beta(z)$ is holomorphic for $z \in U$, then applying this to the harmonic bundle from the previous proposition gives us a new harmonic bundle with holomorphic structure $\bar{\partial}$, with Higgs field given by
\begin{equation*}
\phi(z) = \left( \begin{matrix} 0 & 0 \\ \frac{1}{2} z^\beta & 0 \end{matrix} \right)  z^{-1} dz
\end{equation*}
and harmonic metric
\begin{equation}\label{eqn:twisted-cone-metric}
k_{\beta, \theta}(r) = \left( \begin{matrix} \frac{r^\beta}{2 \theta} (r^{-\theta} - r^\theta) & 0 \\ 0 & \frac{2 \theta}{r^\beta (r^{-\theta} - r^\theta)} \end{matrix} \right) = \left( \begin{matrix} - \frac{r^\beta}{\theta} \sinh(\theta \log r) & 0 \\ 0 & -\frac{\theta}{r^\beta \sinh(\theta \log r)} \end{matrix} \right) .
\end{equation}
Let $\{ w^{1,0}, w^{0,1} \}$ be a basis for the holomorphic sections corresponding to the direct sum $E^{1,0} \oplus E^{0,1}$, related to the basis used in the previous proof by the gauge transformation $g_\beta$. When $\beta = 1$, the growth rate of the sections with respect to the metric \eqref{eqn:twisted-cone-metric} is now
\begin{equation}\label{eqn:twisted-growth-rate}
| w^{1,0}|_{k_{1, \theta}} \sim r^{\frac{1}{2}(1-\theta)}, \quad | z w^{0,1} |_{k_{1, \theta}} \sim r^{\frac{1}{2}(1+\theta)} .
\end{equation}
This defines the weights used in \eqref{eqn:mcowen-weight-divisors} below. Therefore we see that the metric \eqref{eqn:twisted-cone-metric} on $E^{1,0} \oplus E^{0,1}$ corresponds to the growth conditions \eqref{eqn:twisted-growth-rate} for sections $w^{1,0}$ of $E^{1,0}$ and $z w^{0,1}$ of $E^{0,1}[-p]$. This explains why the growth conditions given by \eqref{eqn:mcowen-weight-divisors} on the bundle $K^{\frac{1}{2}} \oplus K^{-\frac{1}{2}}[-D]$ give the correct weights to construct a metric on $K^{\frac{1}{2}} \oplus K^{-\frac{1}{2}}$ in the proof of Corollary \ref{cor:mcowen} below. Define
\begin{equation*}
v^{0,1} := z^\beta w^{0,1} + \theta \cotanh(\theta \log r) w^{1,0} .
\end{equation*}
An analogous calculation to the previous proof shows that $\{ w^{1,0}, v^{0,1} \}$ is a basis for the $d''$-holomorphic sections, and a basis for the flat sections is given by
\begin{align*}
s_1 & = z^{-\frac{1}{2} (\beta + \theta)} (\theta w^{1,0} + v^{0,1}) \\
s_2 & = z^{-\frac{1}{2} (\beta - \theta)} (\theta w^{1,0} - v^{0,1}) .
\end{align*}
When $\beta = 1$ then we can explicitly compute the monodromy $\rho_\theta' : \mathbb{Z} \rightarrow \SL(2, \C)$ of a loop around the origin, which maps a generator of $\mathbb{Z}$ to
\begin{equation*}
\left( \begin{matrix} e^{-i \pi(1+\theta)} & 0 \\ 0 & e^{-i \pi (1-\theta)} \end{matrix} \right) = - \left( \begin{matrix} e^{-\pi i \theta} & 0 \\ 0 & e^{\pi i \theta} \end{matrix} \right) .
\end{equation*}
Returning to the case of a general $\beta$, we can apply a gauge transformation to the previous basis of flat sections to obtain a new basis
\begin{align*}
u_{\beta, \theta}^{1,0} & := \frac{1}{2} \left( z^{-\frac{1}{2} (\beta + \theta)} + z^{-\frac{1}{2}(\beta - \theta)} \right) w^{1,0} + \frac{1}{2\theta} \left( z^{-\frac{1}{2}(\beta+ \theta)} - z^{-\frac{1}{2}(\beta - \theta)} \right) v^{0,1} \\
u_{\beta, \theta}^{0,1} & := \frac{\theta}{2} \left( z^{-\frac{1}{2} (\beta + \theta)} - z^{-\frac{1}{2}(\beta - \theta)} \right) w^{1,0} + \frac{1}{2} \left( z^{-\frac{1}{2}(\beta+ \theta)} + z^{-\frac{1}{2}(\beta - \theta)} \right) v^{0,1}  .
\end{align*}
Note that as $\theta \rightarrow 0$, these sections converge to
\begin{align*}
u_{\beta, 0}^{1,0} & = z^{-\frac{1}{2} \beta} \left( w^{1,0} - \frac{1}{2} \log z v^{0,1} \right) \\
u_{\beta, 0}^{0,1} & = z^{-\frac{1}{2} \beta} v^{0,1} ,
\end{align*}
which is the tensor product of Simpson's rank $2$ local system (cf. \cite[p746]{Simpson90}) with Simpson's rank one local system with weight $-\frac{1}{2} \beta$ (cf. \cite[p745]{Simpson90}).

By applying the nonabelian Hodge theorem to this construction, we obtain a new proof of a theorem of McOwen \cite{McOwen88}.

\begin{corollary}[McOwen]\label{cor:mcowen}
Let $\bar{X}$ be a compact Riemann surface, and fix a smooth Riemannian metric $g$ on $X$ with constant Gauss curvature $K_g \equiv -1$. Let $X = \bar{X} \setminus \{p_1, \ldots, p_m\}$ and fix a cone angle $2 \pi \theta_j$ at each $p_j$. Suppose also that $2g-2 + m - \sum_{j=1}^m \theta_j > 0$. Then there exists a metric $\hat{g}$ conformal to $g$ such that for each $j=1, \ldots, m$ we have $\hat{g} / g = O(r^{2 (\theta_j-1)})$ as $r \rightarrow 0$, where $r$ is the distance to the puncture $p_j$.
\end{corollary}

\begin{proof}
First we define the filtered Higgs bundle as follows. Let $K$ denote the canonical bundle of the compact surface $\bar{X}$, and choose a line bundle $K^{\frac{1}{2}} \rightarrow \bar{X}$ so that $\left( K^{\frac{1}{2}} \right)^2 = K$. Let $D = p_1 + \cdots + p_n$ denote the effective divisor corresponding to the marked points, let $E_0 = K^{\frac{1}{2}} \oplus K^{-\frac{1}{2}}[-D]$ be a bundle over $\bar{X}$ and let $E$ denote the restriction to $X$. Note that $(K^{\frac{1}{2}})^* \otimes K^{-\frac{1}{2}}[-D] \otimes K[D]$ is trivial, and define the Higgs field
\begin{equation*}
\phi = \frac{1}{2} \left( \begin{matrix} 0 & 0 \\ 1 & 0 \end{matrix} \right) \in H^0(\End(E_0) \otimes K[D]) .
\end{equation*}
For each $\alpha \in [0, 1)$, define the divisors
\begin{align}\label{eqn:mcowen-weight-divisors}
\begin{split}
D_1(\alpha) & = \sum_{j=1}^m \varepsilon_j p_j, \quad \text{where} \, \, \varepsilon_j = \begin{cases} 0 & \text{if} \,  \alpha < \frac{1}{2} ( 1 - \theta_j ) \\ 1 & \text{if} \,  \alpha \geq \frac{1}{2} (1-\theta_j) \end{cases} \\
D_2(\alpha) & = \sum_{j=1}^m \varepsilon_j' p_j \quad \text{where} \, \, \varepsilon_j' = \begin{cases} 0 & \text{if} \,  \alpha < \frac{1}{2} (1 + \theta_j) \\ 1 & \text{if} \,  \alpha \geq \frac{1}{2} (1 + \theta_j) \end{cases}
\end{split}
\end{align}
Note that if $\theta_j \in (0,1)$ (i.e. the cone angle is between $0$ and $2\pi$) then these are both trivial when $\alpha = 0$. For each $\alpha \in [0, 1)$, the associated extension of $E$ across the punctures is then given by
\begin{equation*}
E_\alpha = K^{\frac{1}{2}}[-D_1(\alpha)] \oplus K^{-\frac{1}{2}}[-D_2(\alpha)-D] .
\end{equation*}
In particular, we see that $E_0 = K^{\frac{1}{2}} \oplus K^{-\frac{1}{2}}[-D]$. Note that in a neighbourhood of a single puncture, this reduces to the description of the filtered bundle associated to the local model with metric \eqref{eqn:twisted-cone-metric} and growth rate \eqref{eqn:twisted-growth-rate}. The algebraic degree of the filtered bundle is
\begin{equation*}
\deg (E, \{ E_{\alpha} \}) = \deg E_0 + \sum_{j=1}^m \frac{1}{2} ( 1 - \theta_j ) + \sum_{j=1}^m \frac{1}{2} (1 + \theta_j) = -m + m = 0.
\end{equation*}
The only filtered Higgs subbundle is $(K^{-\frac{1}{2}}[-D], \{ K^{-\frac{1}{2}}[-D_2(\alpha)-D] \}_{\alpha} )$, which has degree
\begin{align*}
\deg(K^{-\frac{1}{2}}[-D], \{ K^{-\frac{1}{2}}[-D_2(\alpha)-D] \}_{\alpha} ) & = \deg K^{-\frac{1}{2}}[-D] + \sum_{j=1}^m \frac{1}{2} (1 + \theta_j) \\
 & = 1-g -\frac{1}{2} m + \sum_{j=1}^m \frac{1}{2} \theta_j .
\end{align*}
Therefore the algebraic stability condition 
\begin{equation*}
\slope (K^{-\frac{1}{2}}[-D], \{ K^{-\frac{1}{2}}[-D_2(\alpha)-D] \}_{\alpha} ) <  \slope (E, \{ E_{\alpha} \})
\end{equation*}
reduces to
\begin{equation}\label{eqn:stability-constraint}
1-g - \frac{1}{2} m + \sum_{j=1}^m \frac{1}{2} \theta_j < 0 \quad \Leftrightarrow \quad 2g-2 + m - \sum_{j=1}^m \theta_j > 0  .
\end{equation}
 
Since the filtered Higgs bundle $(E, \phi, \{ E_\alpha\})$ is algebraically stable, then Theorem \ref{thm:nonabelian-hodge} guarantees the existence of a Hermitian metric $h$ solving $F_{h} + [\phi, \phi^*] = 0$ such that the growth rate of the sections on $K^{\frac{1}{2}} \oplus K^{-\frac{1}{2}}[-D]$ is given by \eqref{eqn:twisted-growth-rate}. Equivalently, the induced metric $\hat{h}$ on $K^{\frac{1}{2}} \oplus K^{-\frac{1}{2}}$ is asymptotic to the model cone metric $k_{\beta, \theta_j}$ from \eqref{eqn:twisted-cone-metric} with $\beta = 1$ around each marked point $p_j$.

Let $g = e^{2u}|dz|^2$ be a hyperbolic metric on the compact surface $T \bar{X} = K^{-1}$ with constant scalar curvature $K_g \equiv -1$, and let $h = \left( \begin{matrix} e^{-u} & 0 \\ 0 & e^{u} \end{matrix} \right)$ be the induced metric on $K^{\frac{1}{2}} \oplus K^{-\frac{1}{2}}$. Let $\lambda : X \rightarrow \R$ be the conformal factor such that
\begin{equation*}
\hat{h} = \left( \begin{matrix} e^{-u-\lambda} & 0 \\ 0 & e^{u + \lambda} \end{matrix} \right) .
\end{equation*}
Then the self-duality equations $F_{\hat{h}} + [\phi, \phi^*] = 0$ imply that the metric $\hat{g} = e^{2u + 2 \lambda} |dz|^2$ on the punctured surface also has constant scalar curvature $K_{\hat{g}} \equiv -1$ (cf. Example 1.5 of \cite{Hitchin87}). Since $\hat{h}$ is asymptotic to $k_{1, \theta_j}$ near $p_j$, then in a local coordinate with distance $r$ from the marked point $p_j$, the metric $\hat{g}$ is asymptotic to 
\begin{equation*}
\frac{\hat{g}}{g} = e^{2 \lambda} \sim  \frac{\theta_j^2}{r^2 \sinh^2(\theta_j \log r)} = r^{2(\theta_j-1)} \left( \frac{4 \theta_j^2}{(r^{2\theta_j} - 1)^2} \right) \sim r^{2(\theta_j - 1)} .
\end{equation*}
Therefore $\hat{g}$ is a metric with constant curvature $K_{\hat{g}} \equiv -1$ which satisfies the conditions of the theorem.
\end{proof}

\section{The limiting bundle as the cone angle converges to zero}\label{sec:convergence}

In this section we show that for a fixed algebraically stable parabolic Higgs bundle $(E, \phi, \{ E_\alpha \})$ on a punctured surface, the harmonic metric depends analytically on the weights and the Higgs bundle (Theorem \ref{thm:analytic-parametrisation}). A special case of Theorem \ref{thm:analytic-parametrisation} gives a Higgs bundle proof of a theorem of Judge on the analytic dependence of hyperbolic cone and cusp metrics within a given conformal class (Corollary \ref{cor:Higgs-Judge}). 

Throughout this section we restrict attention to weights and filtrations such that the algebraic degree \eqref{eqn:algebraic-degree} is zero. Fix $p > 2n = 2 \dim_\C X$ and let $\delta > 0$ be small enough so that Proposition \ref{prop:Fredholm-index-zero} applies. Let $\mathcal{H} := \mathcal{H}_\delta^{2,p}$ denote the space of metrics on the bundle $E \rightarrow X$ from \eqref{eqn:weighted-space-metrics}.

When considering varying weights, it is more convenient to use the description of a parabolic bundle consisting of a set of weights and a filtration $\{ E_{p, \alpha} \}$ of the fibre over each marked point, which together with the residue of the Higgs field determines a model metric \eqref{eqn:model-metric}. We use $\mu_0 = \{ \alpha_1, \ldots, \alpha_n\}$ to denote the set of weights of the parabolic structure, i.e. the set of values of $\alpha$ in $[0,1)$ where $\dim_\C \Gr_{\alpha}(E_{p,0}) > 0$. Since algebraic stability is an open condition (cf. \cite{Thaddeus02}) then the parabolic Higgs bundle remains stable for all Higgs bundles $(\bar{\partial}_A, \phi)$, weights $\mu$ and filtrations $\{ E_{p, \beta} \}$ compatible with $\mu$ and $\Res_p \phi$ in a neighbourhood of the initial parabolic bundle $(\bar{\partial}_{A_0}, \phi_0, \mu_0, \{ E_{p, \alpha} \})$. Note that in Corollary \ref{cor:Higgs-Judge} below we fix the Higgs bundle and so the residue $\Res_p \phi$ is fixed, however as the weights change then the dimensions of each $\Gr_\alpha(E_{p,0})$ may change, so the weight filtration may change and therefore the logarithmic terms in the model metric \eqref{eqn:model-metric} may change. An example of this is given by the local model studied in Proposition \ref{prop:rank-2-local-study} where the weight filtration is trivial for $\theta \neq 0$ (since the weights are distinct) and non-trivial in the limit as $\theta \rightarrow 0$, hence the appearance of logarithmic terms in the local model \eqref{eqn:Simpson-rank2-local-model} studied by Simpson. Corollary \ref{cor:Higgs-Judge} below explains the connection between this phenomena and hyperbolic cone metrics converging to a cusp metric on a punctured surface.

From the perspective of flat connections and representations, when the weight filtration is trivial then the associated flat connection $D_{\mu_0}$ determines a representation $\rho_{\mu_0} : \pi_1(X_0) \rightarrow G$ with elliptic holonomy around the marked point, and parabolic terms appear in the holonomy when the weight filtration becomes non-trivial, exactly as described by the rank $2$ local model in the previous section.

Given a punctured surface $X = \bar{X} \setminus \{p_1, \ldots, p_n \}$ and a bundle $E \rightarrow X$, let $\tilde{\mathcal{B}}$ denote the space of parabolic Higgs bundles of algebraic degree zero on $E$. Given a fixed algebraically stable parabolic Higgs bundle $(\bar{\partial}_{A_0}, \phi_0, \mu_0, \{ E_{p, \alpha} \})$, let $\tilde{U}$ be a neighbourhood of $(\bar{\partial}_{A_0}, \phi_0, \mu_0, \{ E_{p, \alpha} \})$ in $\tilde{\mathcal{B}}$, such that any $(\bar{\partial}_A, \phi, \mu, \{ E_{p, \beta} \} ) \in \tilde{U}$ is algebraically stable. Theorem \ref{thm:nonabelian-hodge} determines a map $\tilde{U} \rightarrow \mathcal{H}$. 

Let $\mathcal{B}$ denote the space of triples $(\bar{\partial}_A, \phi, \mu)$ such that there exists a filtration $\{ E_{p, \beta} \}$ for which $(\bar{\partial}_A, \phi, \mu, \{ E_{p, \beta} \} )$ is algebraically stable. There is a map $\tilde{\mathcal{B}} \rightarrow \mathcal{B}$ obtained by forgetting the filtration $\{E_{p, \beta} \}$. Let $U$ be the image of $\tilde{U}$ under this map. Given a choice of metric $h_p$ on $E_{p, 0}$ used in the definition of model metric \eqref{eqn:model-metric}, any two filtrations compatible with $(\bar{\partial}_A, \phi, \mu)$ are related by a change of coordinates which is unitary with respect to $h_p$. Let $\mathcal{H}_0$ denote the quotient of the space of metrics $\mathcal{H}$ by coordinate changes which are unitary with respect to $h_{\mu_0}$. Then the map $\tilde{U} \rightarrow \mathcal{H}$ descends to a well-defined map $U \rightarrow \mathcal{H}_0$. 

Since $(\bar{\partial}_{A_0}, \phi_0)$ is stable and hence irreducible, then one can use the inverse function theorem in analogy with \cite[Ch. VII.3]{Kobayashi87} to define a local diffeomorphism $K : \Omega^{0,1}(\End(E)) \oplus \Omega^{1,0}(\End(E)) \rightarrow \Omega^{0,1}(\End(E)) \oplus \Omega^{1,0}(\End(E))$ with the property that that $(a, \varphi)$ solves the nonlinear equation $\bar{\partial}_{A_0 + a} (\phi_0 + \varphi) = 0$ if and only if $K(a, \varphi)$ solves the linearisation of the slice equations. The explicit form of $K$ (cf. \cite[(3.10)]{Kobayashi87}) shows that it is analytic in $(a, \varphi)$, and so for the definition of the operator $F : U \times \hat{L}_\delta^{2,p} \rightarrow \hat{L}_\delta^p$ in \eqref{eqn:perturbed-HE} below one can define precisely what it means for $F$ to be an analytic map of Banach spaces.

\begin{theorem}\label{thm:analytic-parametrisation}
The map $U \mapsto \mathcal{H}_0$ given by $(a, \varphi, \mu) \mapsto h_\mu$ is analytic. 
\end{theorem}

\begin{proof}
Following the idea of \cite{Judge98} for hyperbolic metrics on a punctured surface, the goal of the proof is to use the implicit function theorem to construct the Hermitian-Einstein metric $h_\mu$ for weights $\mu$ near $\mu_0$ and Higgs bundles $(\bar{\partial}_A, \phi)$ near $(\bar{\partial}_{A_0}, \phi_0)$, in which case the analytic dependence of the metric on $(\bar{\partial}_A, \phi, \mu)$ follows immediately from the implicit function theorem and the fact that the equation \eqref{eqn:perturbed-HE} for the metric depends analytically on $(\bar{\partial}_A, \phi, \mu)$. It is important to note that we are not using the full power of the implicit function theorem, since we already know that Hermitian-Einstein metrics exist and are unique by the nonabelian Hodge theorem. Instead we only use the part of the theorem which shows that the solution depends analytically on the other variables. This still requires showing that the linearised operator \eqref{eqn:linearised-operator} defines an isomorphism of Banach spaces, which we prove below.


Following Simpson's notation from \cite{Simpson92}, with respect to the fixed Higgs bundle $(\bar{\partial}_{A_0}, \phi_0)$, let $D'' = \bar{\partial}_{A_0} + \phi_0 : \Omega^0(\End(E)) \rightarrow \Omega^1(\End(E))$. With respect to a metric $h$, we use $D_h'$ to denote the operator $\partial_{A_0, h} + \phi_0^{*_{h}} : \Omega^0(\End(E)) \rightarrow \Omega^1(\End(E))$ and $D_h = D'' +D_h'$ to denote the associated connection. Then 
\begin{equation*}
D_h^2 = D'' D_h' + D_h' D'' = F_{A_0, h} + [\phi_0, \phi_0^{*_{h}}] .
\end{equation*}
Given two metrics $h_1$ and $h_2$ related by $h_2 = h_1 k$ for $k = g^* g$, $g \in \mathcal{G}^\C$, the change in the curvature is given by
\begin{equation*}
g^{-1} \left( F_{A_0, h_2} + [\phi_0, \phi_0^{*_{h_2}}] \right) g = D'' ( k^{-1} D_{h_1}' k ) + F_{A_0, h_1} + [\phi_0, \phi_0^{*_{h_1}}] .
\end{equation*}

Now suppose that $h_{A_0, \mu_0}$ is the Hermitian-Einstein metric from the nonabelian Hodge theorem for the algebraically stable Higgs pair $(\bar{\partial}_{A_0}, \phi_0)$ with weight $\mu_0$. Then $F_{A_0, h_{\mu_0}} + [\phi_0, \phi_0^{*_{h_{\mu_0}}}] = 0$, and given $(\bar{\partial}_A, \phi, \mu) \in U$ with $(\bar{\partial}_A, \phi) = (\bar{\partial}_{A_0} + a, \phi_0 + \varphi)$, we want to solve $F_{A, h_\mu} + [\phi, \phi^{*_{h_\mu}}] = 0$ for $h_{\mu} = h_{\mu_0} k_{\mu_0, \mu}$. The above equation shows that
\begin{equation}\label{eqn:nonlinear-equation}
i \Lambda D''(k_{\mu_0, \mu}^{-1} D_{\mu_0}' k_{\mu_0, \mu}) = i \Lambda \left( g^{-1} \left( F_{A_0, {h_{\mu}}} + [\phi_0, \phi_0^{*_{h_{\mu}}}] \right) g - \left( F_{A_0, h_{\mu_0}} + [\phi_0, \phi_0^{*_{h_{\mu_0}}}] \right) \right)  .
\end{equation}
and a standard calculation shows that
\begin{multline*}
\left( F_{A, h_\mu} + [\phi, \phi^{*_{h_\mu}}] \right) - \left( F_{A_0, {h_\mu}} + [\phi_0, \phi_0^{*_{h_\mu}}] \right) \\
 = -\bar{\partial}_{A_0} a^{*_{h_\mu}} + \partial_{A_0, h_\mu} a + [\phi_0, \varphi^{*_{h_\mu}}] + [\varphi, \phi_0^{*_{h_\mu}}] - [a, a^{*_{h_\mu}}] + [\varphi, \varphi^{*_{h_\mu}}] \\
  = D'' (-a^{*_{h_\mu}} + \varphi^*) + D_{h_\mu}' (a + \varphi) - [a, a^{*_{h_\mu}}] + [\varphi, \varphi^{*_{h_\mu}}] 
\end{multline*}
Therefore, if $F_{A_0, h_{\mu_0}} + [\phi_0, \phi_0^{*_{h_{\mu_0}}}] = 0 = F_{A, h_\mu} + [\phi, \phi^{*_{h_\mu}}]$, then 
\begin{equation*}
i \Lambda D''(k_{\mu_0, \mu}^{-1} D_{\mu_0}' k_{\mu_0, \mu}) + i g^{-1} \Lambda \left( D'' (-a^{*_{h_\mu}} + \varphi^{*_{h_\mu}}) + D_{h_\mu}' (a + \varphi) - [a, a^{*_{h_\mu}}] + [\varphi, \varphi^{*_{h_\mu}}]  \right) g = 0 .
\end{equation*}
Simpson's K\"ahler identities (cf. \cite[p15]{Simpson92}) show that $i \Lambda D'' = (D_{\mu_0}')^*$ on $1$-forms, and so the equation becomes
\begin{equation}\label{eqn:laplace-nonlinear-equation}
g (D_{\mu_0}')^* (k_{\mu_0, \mu}^{-1} D_{\mu_0}' k_{\mu_0, \mu}) g^{-1} + i \Lambda \left( D'' (-a^{*_{h_\mu}} + \varphi^{*_{h_\mu}}) + D_{h_\mu}' (a + \varphi) - [a, a^{*_{h_\mu}}] + [\varphi, \varphi^{*_{h_\mu}}]  \right)  = 0 .
\end{equation}

For each weight $\mu$, we can write the change of metric $k_{\mu_0, \mu}$ in terms of the model metric $h_{mod, \mu}$ defined in \eqref{eqn:model-metric}
\begin{equation*}
e^{u} = h_{mod, \mu}^{-1} h_{mod, \mu_0} k_{\mu_0, \mu} .
\end{equation*}
Note that $\mu = \mu_0$ implies $u = 0$. Now choose $\delta > 0$ small enough so that Proposition \ref{prop:Fredholm-index-zero} applies and choose $p > \dim_\R X$. Now define the operator $F : U \times \hat{L}_\delta^{2,p} \rightarrow \hat{L}_\delta^p$ by
\begin{multline}\label{eqn:perturbed-HE}
F(a, \varphi, \mu, u) = g (D_{\mu_0}')^* (e^{-u} h_{mod,\mu}^{-1} h_{mod, \mu_0} D_{\mu_0}' (h_{mod, \mu_0}^{-1} h_{mod, \mu} e^{u})) g^{-1} \\
 + i \Lambda \left( D'' (-a^{*_{h_\mu}} + \varphi^{*_{h_\mu}}) + D_{h_\mu}' (a + \varphi) - [a, a^{*_{h_\mu}}] + [\varphi, \varphi^{*_{h_\mu}}]  \right) ,
\end{multline}
where to simplify the notation we use $h_\mu = h_{\mu_0} k_{\mu_0, \mu} = h_{\mu_0} h_{mod, \mu_0}^{-1} h_{mod, \mu} e^u$ in the second line of the above equation. 

Since $h_{\mu_0}$ is Hermitian-Einstein, then Lemmas \ref{lem:perturbation-bounded} and \ref{lem:perturbed-metric} show that $u \in \hat{L}_\delta^{2,p}$ implies that $g (D_{\mu_0}')^* (k_{\mu_0, \mu}^{-1} D_{\mu_0}' (k_{\mu_0, \mu})) g^{-1} \in \hat{L}_\delta^p$ and so $F(a, \varphi, \mu, u) \in \hat{L}_\delta^p$. The Fr\'echet derivative of \eqref{eqn:perturbed-HE} with respect to $u$ at $\mu = \mu_0$ and $(u, a, \varphi) = 0$ is the operator
\begin{equation}\label{eqn:linearised-operator}
(D_{\mu_0}')^* D_{\mu_0}' : \hat{L}_\delta^{2,p} \rightarrow \hat{L}_\delta^p .
\end{equation}
Since the initial parabolic Higgs bundle is algebraically stable, then it is irreducible and so the restriction of $D''$ to $\tilde{L}_\delta^{2,p}$ is injective, therefore Proposition \ref{prop:Fredholm-index-zero} shows that $(D_{\mu_0}')^* D_{\mu_0}'$ is injective. Proposition \ref{prop:Fredholm-index-zero} also shows that $(D_{\mu_0}')^* D_{\mu_0}'$ has index zero, therefore it is surjective onto $\tilde{L}_\delta^p$. We can now apply the analytic implicit function theorem for Banach spaces (cf. \cite{Berger77}) to show that the equation $F(a, \varphi, \mu, u) = 0$ has a unique solution given by $u = u(a, \varphi, \mu)$ in a neighbourhood of $(0, 0, \mu_0, 0)$ and that $u(a, \varphi, \mu)$ depends analytically on the weight $\mu$ and the perturbation $(a, \varphi)$ of the Higgs bundle. Moreover, Theorem \ref{thm:nonabelian-hodge} shows that $u(a, \varphi, \mu)$ is bounded.
\end{proof}

If we consider a rank $2$ Higgs bundle $(E, \phi)$ at a Fuchsian point in the moduli space of stable Higgs bundles as in \cite[Sec. 11]{Hitchin87} (Higgs bundles over compact surfaces) and \cite{Biswas97-1} (parabolic Higgs bundles) then we can use the previous theorem together with the proof of Corollary \ref{cor:mcowen} to obtain the following theorem of Judge \cite{Judge98}.

\begin{corollary}\label{cor:Higgs-Judge}
Let $\bar{X}$ be a compact Riemann surface of genus $g$ with $m$ marked points $\{ p_1, \ldots, p_m \}$ such that $2g-2+m > 0$ and let $X = \bar{X} \setminus \{p_1, \ldots, p_m\}$. Let $g_0$ be the unique complete metric on $X$, isometric to a model cusp on a neighbourhood of each marked point $p_1, \ldots, p_m$. For each choice $\theta = \{\theta_1, \ldots, \theta_m\}$ of cone angles at $p_1, \ldots, p_m$ such that $2g-2+m-\sum_{k=1}^m \theta_k > 0$, let $g_\theta$ be the unique metric on $X$ which is conformal to $g_0$ such that $(X, g_\theta)$ is isometric to a model cone of angle $\theta_j$ on a neighbourhood of each marked point $p_j$. Then the map $\theta \mapsto g_\theta$ is analytic.
\end{corollary}


\end{document}